\documentclass[11pt]{amsart}
\usepackage{amssymb,amsmath}
\usepackage{esint} 
\usepackage{graphicx}
\usepackage{bm,bbm}
\usepackage{color}
\usepackage{tikz}
\usepackage{makecell}  
\usetikzlibrary{snakes}
\usepackage{pgfplots}
\usepackage{array,blkarray}
\usepackage{bigdelim}
\usepackage{listings}
\usepackage{hhline}
\usepackage{arydshln}
\usepackage{multirow}
\usepackage{enumitem}

\usetikzlibrary{external}
\usepackage{mathtools}  
\newtheorem{theorem}{Theorem}

\theoremstyle{definition}

\newtheorem{example}[theorem]{Example}
\theoremstyle{lemma}
\newtheorem{lemma}[theorem]{Lemma}

\theoremstyle{remark}
\newtheorem{remark}[theorem]{Remark}

\numberwithin{theorem}{section}
\numberwithin{equation}{section}
\numberwithin{table}{section}
\numberwithin{figure}{section}

\newcommand{\uo}{\ensuremath{u_0}}
\newcommand{\uB}{\ensuremath{u^{[1]}}}
\newcommand{\duB}{\ensuremath{{\dot u}^{[1]}}}
\newcommand{\uC}{\ensuremath{u^{[2]}}}

\def\P{P_0}
\def\one{\mathbbm{1}_{\Omega_c}}   
\def\ttau{\tfrac \tau 2}
\DeclareMathOperator{\ee}{e}
\newcommand{\calF}{\ensuremath{F}}
\newcommand{\calG}{\ensuremath{G}}
\def\R{\mathbb{R}}
\definecolor{myBlue2}{RGB}{113,104,238} 
\definecolor{myBlue3}{RGB}{30,144,255} 
\definecolor{myGreen2}{RGB}{69,169,0} 
\definecolor{myGreen3}{RGB}{154,205,50} 
\definecolor{myRed2}{RGB}{165,42,42} 

\DeclareMathOperator{\id}{id}
\DeclareMathOperator{\ddiv}{div}

\DeclareMathOperator{\range}{im}

\newcommand{\calA}{\ensuremath{\mathcal{A}} }
\newcommand{\calAo}{\ensuremath{\mathcal{A}_0}}
\newcommand{\calB}{\ensuremath{\mathcal{B}} }
\newcommand{\calD}{\ensuremath{\mathcal{D}} }
\newcommand{\cL}{\ensuremath{\mathcal{L}} }

\newcommand{\calS}{\ensuremath{\mathcal{S}} }

\def\ds{\,\text{d}s}

\def\dx{\,\text{d}x}
\def\dxi{\,\text{d}\xi}

\begin{document}
\title[Splitting Methods for Constrained Diffusion-Reaction Systems]{Splitting Methods for Constrained\\ Diffusion-Reaction Systems}
\author[]{R. Altmann$^{*}$ and A. Ostermann$^{\dagger}$}
\address{${}^{*}$ Institut f\"ur Mathematik MA4-5, Technische Universit\"at Berlin, Stra\ss e des 17.~Juni 136, 10623 Berlin, Germany}
\address{${}^{\dagger}$ Department of Mathematics, University of Innsbruck, Technikerstra\ss e 13, 6020 Innsbruck, Austria}
\email{raltmann@math.tu-berlin.de, alexander.ostermann@uibk.ac.at}
\thanks{${}^{*}$ The work of R. Altmann was supported by a postdoc fellowship of the German Academic Exchange Service (DAAD)}

\date{\today}
\begin{abstract}
We consider Lie and Strang splitting for the time integration of constrained partial differential equations with a nonlinear reaction term.
Since such systems are known to be sensitive with respect to perturbations, the splitting procedure seems promising as we can treat the nonlinearity separately.
This has some computational advantages, since we only have to solve a linear constrained system and a nonlinear ODE.
However, Strang splitting suffers from order reduction which limits its efficiency.
This is caused by the fact that the nonlinear subsystem produces inconsistent initial values for the constrained subsystem.
The incorporation of an additional correction term resolves this problem without increasing the computational cost.
Numerical examples including a coupled mechanical system illustrate the proven convergence results.
\end{abstract}

\maketitle
\setcounter{tocdepth}{2}
%
%
\section{Introduction}
Splitting methods for evolution equations have been studied extensively in the past years.
The use of such methods for time integration allows one to split the problem into subsystems which can be integrated more efficiently or sometimes even exactly \cite{HunV95, McLQ02, HunV03}.
Moreover, splitting methods perform well in the preservation of geometric properties, which is one of the main keys for numerical integration \cite{HaiLW10}.
For reaction-diffusion equations with a linear elliptic diffusion operator, several contributions to splitting methods can be found in the literature, e.g., \cite{Des01, HanKO12, EinO15}. 

In this paper, we consider nonlinear evolution equations of diffusion-reaction type which have an additional constraint.
The diffusion is assumed to be a linear elliptic differential operator, whereas the reaction is nonlinear but smooth.
Because of the constraint, we deal with so-called {\em partial differential-algebraic equations} (PDAEs) which generalize the concept of differential-algebraic equations \cite{KunM06, LamMT13} to the infinite dimensional case.
Since these systems are known to be very sensitive (e.g.~with respect to perturbations), we propose to apply splitting methods that reduce the given system to a linear PDAE and a nonlinear ordinary differential equation (ODE).
With this, we only need to solve a nonlinear system in every spatial discretization point instead of a high-dimensional nonlinear PDAE.

It has been observed in various situations that splitting methods suffer from an order reduction if, e.g., non-trivial boundary conditions are prescribed \cite{HanO09,EinO15,EinO16ppt}. 
Since boundary conditions can be seen as a constraint on the dynamics, similar effects are expected for general constraints.
This can also be observed in numerical experiments.
To overcome this kind of order reduction, we introduce a modification similar as in \cite{EinO16ppt}.

In this paper, we consider constrained systems of the form
\begin{subequations}
\begin{alignat*}{5}
  \dot u\ -\ &\calA u\ & -\ & f(u)\ & +\ \calD^- \lambda &=&\ \calF&, \\
  &\calD u& & & &=&\ \calG&
\end{alignat*}
\end{subequations}
on a time interval $[0,T]$ with consistent initial condition $u(0)=\uo$, i.e., $\calD \uo = \calG(0)$.
To enforce the constraint we use the Lagrangian method with multiplier $\lambda$.
As mentioned before, we take a splitting approach in order to eliminate the nonlinearity from the PDAE.
More precisely, we split the problem into a nonlinear unconstrained ODE and a linear PDAE of similar structure as the original problem.
On the time interval $[t_n, t_{n+1}]$ the two subsystems have the form
\begin{align*}
  \dot w_n = f(w_n)- q_n, \qquad
  w_n(0) = u_n,
\end{align*}
which is a nonlinear ODE (the reaction part only), and the linear system
\begin{align*}
  \dot v_n - \calA v_n + \calD^- \lambda = \calF_n + q_n, \qquad
  \calD v_n = \calG_n, \qquad
  v_n(0) = w_n(\tau).
\end{align*}
Note that $u_n$ denotes the approximation to the solution $u(t_n)$, whereas $v_n$ and $w_n$ denote the exact solutions of the two subsystems, and $\calF_n(s):=\calF(t_n+s)$, $\calG_n(s):=\calG(t_n+s)$.
We introduce a correction term $q_n$ which is zero for the classical Lie and Strang splitting but aims to maintain the expected convergence orders of those methods.
With a suitable correction we are able to obtain second-order convergence of Strang splitting.
Such a correction is necessary, since the outcome of the ODE is, in general, not consistent with the given constraint.

The paper is structured as follows.
In Section~\ref{sect_prelim} we provide the required assumptions on the considered PDAE system and present a number of examples with different kinds of constraints which fit into the given framework.
Moreover, we derive a variation-of-constants formula for linear PDAEs of the given structure which is the basis for the subsequent analysis.
The convergence analysis for Lie and Strang splitting are then given in Sections~\ref{sect_lie} and \ref{sect_strang}, respectively.
There we show that Lie splitting converges also without a correction term, whereas Strang splitting requires a correction to guarantee second-order convergence.
In Section~\ref{sect_numerics} we consider three numerical examples.
These include a coupled system of an elastic string and a nonlinear spring-damper system.
Finally, we conclude in Section~\ref{sect_conclusion}.
%
%
\section{Preliminaries}\label{sect_prelim}  
Let $X$ and $Q$ be two Banach spaces and consider a PDE of the form
\[
  \dot u - \calA u - f(u) = \calF
\]
with linear operator $\calA\colon D(\calA) \subseteq X \to X$ and a smooth nonlinearity $f$.
The system is constrained by
\[
 \calD u = \calG.
\]
Here, $\calD\colon X \to Q$ is the linear constraint operator which is assumed to have a right-inverse $\calD^-\colon Q \to X$.
The right-hand sides are given functions $\calF\colon [0,T] \to X$ and $\calG\colon [0,T] \to Q$.
%
We formulate the constrained PDE using the Lagrangian method, i.e., we add a Lagrange multiplier $\lambda\colon [0,T] \to Q$ which enforces the constraint and consider the system
\begin{subequations}
\label{eqn_opDAE}
\begin{alignat}{5}
  \dot u\ -\ &\calA u\ - &\phantom{i}f(u) & +\ \calD^- \lambda\ &=&\ \calF&, \label{eqn_opDAE_a}\\
  &\calD u& & &=&\ \calG&  . \label{eqn_opDAE_b}
\end{alignat}
\end{subequations}
Note that the first equation is given in $X$ while the second equation in formulated in the Banach space $Q$.
Equation \eqref{eqn_opDAE_a} includes the right-inverse of $\calD$ which enables the formulation of the dynamics in $X$ although it is constrained to a subspace.
Note that we consider here the right-inverse in place of the dual operator of $\calD$, which is used when working in the weak setting.

We call such a system a partial differential-algebraic equation, since it generalizes both the concept of a classical DAE and of a PDE.
However, we consider here the semigroup setting in contrast to the weak setting used, e.g., in \cite{EmmM13, LamMT13, Alt15} that corresponds to the weak formulation of the underlying PDE.
We assume that $\calA$ generates an analytic semigroup on the kernel of $\calD$, cf.~Section~\ref{sect_prelim_ass} for the precise assumptions.
The given setting includes the following example.
Further examples with more details are given in Section~\ref{sect_prelim_exp}.
\begin{example}[weighted integral mean]
\label{exp_integralMean}
Consider the semilinear heat equation, i.e., a parabolic PDE with a polynomial nonlinearity, subject to an additional constraint on the integral of the solution.
On the domain $\Omega = (0,1)$ the constraint operator $\calD$ is given by
\[
  \calD\colon X = L^2(\Omega) \to Q = \R, \qquad
  \calD v = \int_0^1 v(x) \sin(\pi x) \dx.
\]
The right-hand side $\calG\colon [0,T] \to \R$ equals the prescribed mean value.
The overall system may have the form
\[
  \dot u - \Delta u - u^2 + \calD^- \lambda = 0, \qquad 
  \calD u = \calG.
\]
In this case, the operator $\calA$ corresponds to the Laplacian (with homogeneous Dirichlet boundary conditions) and $f(u)=u^2$ is a polynomial nonlinearity.
\end{example}
%
%
\subsection{Assumptions}\label{sect_prelim_ass}
We consider an open and bounded domain $\Omega\subseteq \R^d$ with smooth boundary.
Furthermore, the involved operators and right-hand sides should satisfy the following six conditions.
\begin{itemize}[itemsep=4pt]
  %
  \item[(A1)] The nonlinearity $f\colon X\to X$ is two times Fr{\'{e}}chet differentiable and system \eqref{eqn_opDAE} has a unique and bounded solution $u$ on $[0, T]$.
   Moreover, the composition of $f$ with the exact solution satisfies $f(u)\in D(\calA)$.
  %
  \item[(A2)] The initial condition is consistent, i.e., $\calD \uo = \calG(0)$ and $\uo \in D(\calA)$.
  %
  \item[(A3)] The constraint operator $\calD\colon X \to Q$ is linear, onto, and its kernel $X_0 := \ker \calD$ is a closed subspace of $X$.
  \item[(A4)] There exists a right-inverse $\calD^-\colon Q \to X$ such that $\calD \calD^- q = q$ for all $q\in Q$.
\end{itemize}
The existence of a right-inverse implies that $X_0$ is a complemented subspace with projection $\P := \id - \calD^-\calD\colon X\to X_0$ and complement $X^c := \ker \P = \range \calD^-$.
%
This can be seen as follows.
$\P$ is a projection on $X_0$, since $\calD \P x = \calD x - \calD x = 0$ and $\P^2=\P$.
Furthermore, the right-inverse maps to $X^c$, since
\[
  \P \calD^- q
  = \calD^- q - \calD^- q
  = 0.
\]
Thus, the choice of the right-inverse determines $X^c$ as well as the projection $\P$.
In the Hilbert space setting, we may choose $X^c$ as the orthogonal complement of $X_0$.
In any case, we have the decomposition of $X$ into $X = X_0 \oplus X^c$.
\begin{itemize}[itemsep=4pt]
  \item[(A5)] The right-hand side $\calF\colon [0,T] \to X$ is continuous in time and $\calG\colon [0,T] \to Q$ is Lipschitz continuous and satisfies in addition $\calD^-\calG \in D(\calA)$.
  %
  \item[(A6)] Given $\calA\colon D(\calA)\subseteq X\to X$ and its restriction $\calA|_{X_0}\colon D(\calA)\cap X_0 \to X$, we assume that $\calAo:= \P \calA|_{X_0}\colon D(\calA)\cap X_0 \to X_0$ generates an analytic semigroup.
\end{itemize}
\begin{remark}
\label{rem_A}
We do not assume that $\calA$ itself generates a semigroup.
However, if $\calA\colon D(\calA)\subseteq X \to X$ generates an analytic semigroup (e.g.~the Laplacian with homogeneous boundary conditions), then it is sufficient that $X_0$ is an {\em invariant subspace of $\calA$}, meaning that $\calA\colon D(\calA)\cap X_0 \to X_0$, cf.~\cite[Chap.~4.5]{Paz83}. 
\end{remark}
One important consequence of $\calAo$ being an analytic semigroup is the {\em parabolic smoothing property}, i.e., there exists a positive constant $C$ such that
\begin{align}
\label{eqn_smoothing}
  \Vert \ee^{t \calAo} \calAo \Vert
  \le \frac C t, \qquad
  0< t \le T.
\end{align}
%
%
\subsection{Examples}\label{sect_prelim_exp}
We give a number of examples which satisfy the assumptions given in the previous subsection.
Three of these examples are then part of the numerical experiments in Section~\ref{sect_numerics}.
%
\subsubsection{Weighted integral mean}\label{sect_prelim_exp_int}
We reconsider Example~\ref{exp_integralMean} with $\Omega=(0,1)$, $X = L^2(\Omega)$, $Q=\R$, and the constraint operator $\calD v = \int_0^1 v(x) \sin(\pi x) \dx$.
In the example, the operator $\calA$ corresponds to the Laplacian with homogeneous boundary conditions.
Thus, we consider the domain $D(\calA) = H^2(\Omega)\cap H^1_0(\Omega)$ for which it is known that $\calA$ generates an analytic semigroup \cite[Sect.~3.1]{Lun95}.
The kernel of the constraint operator is given by
\[
  X_0
  = \ker \calD
  = \big\{ v \in L^2(\Omega)\ \big|\ \int_0^1 v(x) \sin(\pi x) \dx = 0 \big\}.
\]
Note that $X_0$ is an invariant subspace of $\calA$, since $u\in D(\calA) \cap X_0$ implies
\[
  \calD \calA u
  = \int_0^1 u''(x) \sin(\pi x) \dx
  = - \pi^2 \int_0^1 u(x) \sin(\pi x) \dx
  = 0
\]
and thus $\calA u \in X_0$.
This implies that $\calAo := \calA|_{X_0}$ generates a semigroup, cf.~Remark~\ref{rem_A}.
The corresponding numerical experiment can be found in Section~\ref{sect_numerics_int}.

If $\calA$ corresponds to the Laplacian with periodic boundary conditions, 
then we may also consider the constraint operator $\calD v = \int_0^1 v(x) \dx$, since $u\in D(\calA) \cap X_0$ would then imply
\[
  \calD \calA u
  = \int_0^1 u''(x) \dx
  = u'(1) - u'(0)
  = 0.
\]
%
%
\subsubsection{Specification on a subset}\label{sect_prelim_exp_obs}
We consider the same operator $\calA$ and nonlinearity $f$ as before and add the constraint that the solution $u$ is prescribed on a closed subset $\Omega_0 \subseteq \Omega \subseteq \R^d$ of positive measure.
With $X = L^2(\Omega)$ and $Q = L^2(\Omega_0)$ the constraint operator has the form
\[
  \calD\colon X \to Q, \qquad
  \calD u = u|_{\Omega_0}.
\]
Its kernel is given by
\[
  X_0 = \big\{ u \in L^2(\Omega)\ \big|\ u|_{\Omega_0} =0 \big\}.
\]
Clearly, $X_0$ is an invariant subspace of $\calA$ such that $\calAo := \calA|_{X_0}$ generates a semigroup, cf.~Remark~\ref{rem_A}.
For the numerical results of this example we refer to Section~\ref{sect_numerics_obs}.
%
%
\subsubsection{Constraint of Stokes-type}\label{sect_prelim_exp_div}
The constraint operator $\calD$ may also be a differential operator such as the divergence operator $\ddiv\colon [H^1_0(\Omega)]^3 \to L^2_0(\Omega)$.
Such a constraint appears in the Navier--Stokes as well as in Maxwell's equations.

For the Stokes equations on a bounded domain $\Omega \subseteq \R^3$ with smooth boundary the semigroup property has been shown in \cite{FujK64}, see also \cite[Sect.~3.8]{Hen81}.
For domains with Lipschitz boundary we refer to \cite{Tay00}.
We consider $X=[L^2(\Omega)]^3$.
The subspace $X_0$ contains the divergence-free functions in $X$, i.e.,
\[
  X_0
  = \big\{ u\in [L^2(\Omega)]^3\  \big|\ \ddiv u=0,\ u\cdot n|_{\partial\Omega} = 0 \big\}
\]
with outward normal $n$.
Note that $X_0$ is a closed and proper subspace of $X$.
The domain of $\calAo := -\P \Delta$ which generates an analytic semigroup is given by $D(\calAo) = [H^2(\Omega)]^3 \cap [H^1_0(\Omega)]^3 \cap X_0$.

Since the constraint operator in this form is not defined on the entire space $X$ as assumed in (A3) of Section~\ref{sect_prelim_ass}, we have to consider the divengence in a weaker sense.
For this, consider $\calD=\ddiv\colon X = [L^2(\Omega)]^3 \to [H^1(\Omega)]^*$ with
\[
  \big\langle \ddiv u, v \big\rangle
  := - \int_\Omega u \cdot \nabla v \dx
  \qquad\text{for all } v\in H^1(\Omega).
\]
The kernel of this operator is exactly $X_0$, since $\langle \ddiv u, v \rangle = 0$ for all $v\in H^1(\Omega)$ implies $u\in X_0$
\cite{You13}.
Note, however, that the choice of the weak divergence operator influences the ansatz space for the Lagrange multiplier.
We emphasize that the given setting does not include the Navier--Stokes equations because of the involved nonlinearity.
%
\subsubsection{Coupled mechanical system}\label{sect_prelim_exp_spring}
We consider an example from structural dynamics, namely a flexible string (fixed at its endpoints) which is coupled with a nonlinear spring-damper system, cf.~Figure~\ref{fig_spring}.
The elastic string is modelled in $\Omega=(0,1)$ by the system
\begin{align*}
  &\ddot u = c \Delta u - d_1 \dot u (1+\dot u), \\
  &u(0,t) = u(1,t) = 0, \\ 
  &u(0) = \uo,\ \dot u(0) = v_0
\end{align*}
with $c>0$ and damping parameter $d_1 \ge 0$. 
Note that this string model has a linear stiffness but a nonlinear damping term.
The spring-damper system on the other hand is modelled by
\begin{align*}
  &\ddot q + d_2 \dot q + k(q) = 0, \\    
  &q(0) = q_0,\ \dot q(0) = p_0
\end{align*}
with $d_2\ge 0$ and the nonlinearity $k(q)= k_0(1 - a^2q^2)\, q$.
For this so-called {\em softening spring} we assume $|aq|<1$, i.e., only small amplitudes.
\begin{figure}[h!]
\begin{center}
    \includegraphics{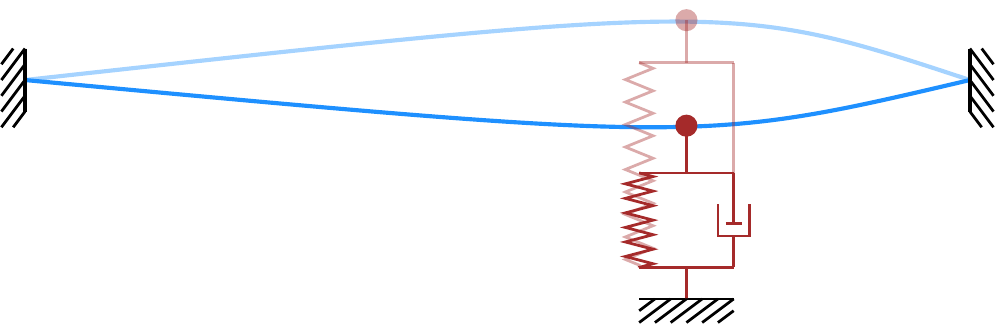}
	\end{center}
\caption{Illustration of the coupling of an elastic string with a spring-damper system.}
\label{fig_spring}
\end{figure}

We write the two systems in their first-order formulation which leads to
\[
  \begin{bmatrix} \dot u \\ \dot v \end{bmatrix}
  = \underbrace{\begin{bmatrix} 0 & I \\ c \Delta & -d_1 I \end{bmatrix}}_{\calA_1}
     \begin{bmatrix} u \\  v \end{bmatrix}
    + \underbrace{\begin{bmatrix} 0 \\  -d_1 v^2 \end{bmatrix}}_{f_1(v)}, \quad
  \begin{bmatrix} \dot q \\ \dot p \end{bmatrix}
  = \underbrace{\begin{bmatrix} 0 & 1 \\ 0 & -d_1 \end{bmatrix}}_{\calA_2}
     \begin{bmatrix} q \\  p \end{bmatrix}
     + \underbrace{\begin{bmatrix} 0 \\  -k(q) + (d_1-d_2)p \end{bmatrix}}_{f_2(q,p)} .
\]
Note that we include the linear term $(d_1-d_2)p$ in the nonlinearity $f_2$ in order to obtain the invariance of $X_0$ later.
In addition, we choose the spaces
\[
  X_1 = H^1_0(\Omega) \times L^2(\Omega), \qquad
  D(\calA_1) = \big( H^2(\Omega)\cap H^1_0(\Omega) \big) \times H^1_0(\Omega), \qquad
  X_2 = \R^2
\]
which imply that $\calA_1$ generates a semigroup \cite[Chap.~VI.3]{EngN00}.
Also $\calA_2$ generates a semigroup, since it corresponds to a linear ODE.

For the coupling of the two subsystems we would like to connect the string at some point $x_c \in (0,1)$ to the spring-damper system, i.e., ask for $u(x_c, t) = q(t)$.
In order to avoid point constraints, we replace this condition by
\[
  \calB u(t) := u(t)|_{\Omega_c} \equiv q(t),
\]
i.e., the spring is connected to the elastic string along a subset $\Omega_c \subseteq (0,1)$.
This subset may be a small neighbourhood of $x_c$.
Including this constraint by means of a Lagrange multiplier, we do not obtain a system of the form \eqref{eqn_opDAE}.
This is caused by the fact that constrained mechanical systems (in the finite-dimensional case) lead to DAEs of index 3 \cite{Sim98} whereas system \eqref{eqn_opDAE} is of index-2 structure.
This means that a spatial discretization of such a system will lead (under certain conditions) to a DAE of index 2 \cite{Alt15}.
For an introduction to the index concept for DAEs, we refer to \cite{Meh15}.

To overcome this modelling issue, we extend the constraint by the so-called {\em hidden constraint}, i.e., we also enforce $\calB v(t) \equiv p(t)$.
In total, this leads to the constraint operator
\[
  \calD\colon X_1 \times X_2 \to L^2(\Omega_c)^2, \qquad
  \calD \begin{bmatrix} u & v & q & p \end{bmatrix}^T
  := \begin{bmatrix} \calB u  - q \one \\  \calB v - p \one \end{bmatrix}
\]
with $\one\in L^2(\Omega_c)$ being the function which is constant with value one.
The overall system then has the form
\begin{align*}
  \dot z
  = \begin{bmatrix} \calA_1 & 0 \\ 0 & \calA_2 \end{bmatrix} z
     +  \calD^- \begin{bmatrix} \lambda \\ \mu \end{bmatrix}
     +  \begin{bmatrix} f_1(v) \\  f_2(q,p) \end{bmatrix}, \qquad
  \calD z = \calG.
\end{align*}
We denote the new linear operator by $\calA$.
It remains to check whether the coupled system fits into the given framework.
For this, we show again that the space $X_0$ is invariant under $\calA$, cf.~Remark~\ref{rem_A}.
Consider any $z_0 \in X_0 = \ker \calD$, i.e., its components satisfy $\calB u_0 = q_0 \one$ and $\calB v_0 = p_0 \one$.
This implies that $u_0$ is constant on $\Omega_c$ and thus, $\calB \Delta u_0 \equiv 0$.
As a result, $z =\ee^{\tau\calAo} z_0$ satisfies $\calD z = 0$ as well and therefore $z\in X_0$.

Numerical results for this example are given in Section~\ref{sect_numerics_spring}.
%
%
\subsubsection{Dirichlet boundary conditions}
Finally, we comment on the inclusion of boundary conditions in the given setting.
As shown in \cite{Alt15}, also Dirichlet boundary conditions may be included in form of a constraint.
However, this requires that function evaluations at the boundary are well-defined (in the sense of traces) and thus, $X=H^1(\Omega)$.
In this case, the constraint operator equals the trace operator with kernel $X_0 = H^1_0(\Omega)$.

The inclusion of (even more general) boundary conditions in combination with splitting methods was already addressed in \cite{EinO15, EinO16ppt}.
There, the space $X=L^p(\Omega)$ is used together with a different formulation without Lagrange multipliers.
%
%
\subsection{Solution formula}\label{sect_prelim_solFormula}
For the convergence analysis in Sections~\ref{sect_lie} and \ref{sect_strang} we make use of a variation-of-constants formula for constrained systems of the form \eqref{eqn_opDAE}.
A corresponding formula for operator DAEs in the weak setting is given in \cite{EmmM13}. 
In the present setting we consider the system without nonlinearity, i.e.,
\begin{subequations}
\begin{alignat*}{5}
  \dot u\ -\ &\calA u\ & +\ \calD^- \lambda\ &=&\ \calF&, \\ 
  &\calD u&  &=&\ \calG&, 
\end{alignat*}
\end{subequations}
where $\calAo:= \P \calA|_{X_0}$ generates an analytic semigroup, cf.~Section~\ref{sect_prelim_ass}.
Consider the unique decomposition of the solution $u\in X$ into
\[
  u = \uB + \uC, \qquad
  \uB \in X_0,\ \uC \in X^c.
\]
With the right-inverse $\calD^-$ applied to the constraint $\calD u = \calG$ we obtain
\[
  \uC(t)
  = \calD^- \calD u(t)
  =\calD^- \calG(t).
\]
Note that $\uC$ is unique but depends on the choice of the complement space $X^c$, i.e., on the choice of the right-inverse.
It remains to find $\uB$.
For this, we project equation \eqref{eqn_opDAE_a} onto the kernel of the constraint which gives
\[
  \duB - \P\calA (\uB + \uC) = \P\calF.
\]
Here we have used $\P\calD^- \lambda = 0$ which follows from $\calD^-\colon Q \to X^c$.
Thus, we obtain with $\calAo=\P\calA|_{X_0}$,
\[
  \duB - \calAo \uB = \P \big( \calA \uC + \calF \big)
\]
for which we can use the variation-of-constants formula for evolution equations, since $\calAo$ is assumed to generate an analytic semigroup.
For consistent initial data $u(0) = \uo$ the initial condition for $\uB$ has the form
\[
  \uB(0) = \uo - \calD^-\calG(0) \in X_0.
\]
Together with the assumptions $\uo\in D(\calA)$ and $\calD^-\calG(0) \in D(\calA)$ this implies $\uB(0)\in D(\calAo)$.
Combining the solution parts, we obtain the formula
\begin{equation}
\label{eqn_solution}
\begin{split}
  u(t)
  &= \calD^-\calG(t) + \uB(t)   \\
  &= \calD^-\calG(t) + \ee^{t\calAo} \big(  \uo - \calD^-\calG(0) \big)
    + \int_0^t \ee^{(t-s)\calAo} \P \Big( \calF(s) + \calA\calD^-\calG(s) \Big) \ds.
\end{split}
\end{equation}
\begin{remark}
\label{rem_inconsistent}
Because of the splitting, we also have to deal with inconsistent initial data $\uo$.
This means that $\uo - \calD^-\calG(0) \not\in X_0$.
In this case, we have to enforce consistency by using $\P \uo \in D(\calAo)$.
For $t>0$ this leads to the formula
\[
  u(t) = \calD^-\calG(t) + \ee^{t\calAo} \P  \uo
    + \int_0^t \ee^{(t-s)\calAo} \P \Big( \calF(s) + \calA\calD^-\calG(s) \Big) \ds.
\]
Note that this implies a discontinuity at $t=0$.
\end{remark}
%
%
\subsection{Lagrange multiplier}\label{sect_prelim_lagrange}
For the sake of completeness we also give the solution formula for the Lagrange multiplier although we are mainly interested in the variable $u$.
The application of the constraint operator $\calD$ to the first equation of the linear system yields
\[
  \lambda
  = \calD \calD^- \lambda
  = \calD \big( \calF - \dot u + \calA u \big)
  = \calD \big( \calF + \calA u \big) - \dot \calG.
\]
%
Note that, since $\lambda$ is not a differential variable (in time), its approximation does not depend on previous values of the multiplier.

\section{Lie Splitting}\label{sect_lie}
Given the assumptions from the previous section, we analyse Lie splitting for constrained systems of the form \eqref{eqn_opDAE}.
We show first-order convergence and discuss the role of the introduced correction term.
%
%
\subsection{Modified Lie splitting}\label{sect_lie_alg}
For its numerical solution we split the system
\begin{subequations}
\label{eqn_lie_gesamt}
\begin{alignat}{5}
  \dot u\ -\ &\calA u\ & -\ & f(u)\ & +\ \calD^- \lambda &=&\ \calF&, \label{eqn_lie_gesamt_a}\\
  &\calD u& & & &=&\ \calG&  \label{eqn_lie_gesamt_b}
\end{alignat}
\end{subequations}
with initial condition $u(0) = \uo$ into the following two subsystems.
First, a nonlinear ODE including the reaction term on the time interval $[t_n, t_{n+1}]$,
\begin{align}
\label{eqn_lie_A}
  \dot w_n = f(w_n) - q_n, \qquad
  w_n(0) = u_n
\end{align}
and second, a linear PDAE of the form
\begin{align}
\label{eqn_lie_B}
  \dot v_n - \calA v_n + \calD^- \lambda = \calF_n + q_n, \qquad
  \calD v_n = \calG_n, \qquad
  v_n(0) = w_n(\tau).
\end{align}
Here $q_n$ denotes a correction term that still has to be chosen.
With $q_n=0$ we obtain the classical Lie splitting.
Recall that $u_n$ denotes the approximation of $u$ at time $t=t_n$, whereas $v_n$ and $w_n$ are the exact solutions of the two subsystems.
In total, one step of Lie splitting is given by
\[
  u_{n+1}
  = \cL_\tau u_n
  := v_n(\tau).
\]
Note that the order of the two subsystems may also be reversed, cf.~the discussion in Section~\ref{sect_lie_rev} below.

It remains to define the correction $q_n$ with which we aim to solve the inconsistency of the initial data in subsystem \eqref{eqn_lie_B} -- at least up to a certain order.
For this, the correction must be smooth, more precisely $q_n(s) \in D(\calA)$.
With assumption (A1), this implies
\begin{align}
\label{eqn_lie_smoothness}
  q_n(s)-f(u(t_n+s)) \in D(\calA).
\end{align}
Note that even $f(u) \not\in D(\calA)$ would be possible as long as the choice of the correction ensures \eqref{eqn_lie_smoothness}.
Moreover, we suggest (although this is not necessary for Lie splitting) to choose a correction that satisfies
\begin{align}
\label{eqn_lie_qn}
  \calD q_n(0) &= \calD f(u(t_n)) +  O(\tau).
\end{align}
One possible choice of the correction satisfying this condition is given by
\begin{align}
\label{eqn_qn_choiceA}
  q_n(s)
  \equiv q_n(0)
  := f(u_n),
\end{align}
i.e., we simply insert the current numerical solution into the nonlinearity.
Since $u_n$ equals the exact solution up to order (at least) one, as we will show below, this implies
\[
  \calD q_n(0)
  = \calD f(u_n)
  = \calD f(u(t_n)) + O(\tau).
\]
A second possible choice for the correction term is given by
\begin{align}
\label{eqn_qn_choiceB}
  q_n(s)
  \equiv q_n(0)
  := f(\calD^-\calG(t_n)),
\end{align}
i.e., we consider only one part of the solution, namely $(I-\P)u(t_n) = \calD^-\calG(t_n) \in X^c$.
However, this choice satisfies \eqref{eqn_lie_qn} only for a special class of constrained systems, as the following lemma shows.
\begin{lemma}
\label{lem_qn}
If the constraint operator $\calD$ and the projection $\P\colon X\to X_0$ satisfy
\[
  \calD f(u)
  = \calD (I-\P) f(u)
  =\calD f( (I-\P) u ),
\]
i.e., the nonlinearity $f$ and the projection $I-\P$ commute under the action of $\calD$, then the choice given in \eqref{eqn_qn_choiceB} satisfies the condition \eqref{eqn_lie_qn} exactly.
\end{lemma}
\begin{proof}
We use the fact that $(I-\P) u(t) = \calD^-\calG(t)$ holds for the exact solution.
The simple calculation
\begin{align*}
  \calD q_n(0)
  = \calD  f( \calD^- \calG(t_n) )  				
  = \calD  f( (I-\P) u(t_n) )   		  				
  = \calD  (I-\P) f( u(t_n) )     	      			
   = \calD  f( u(t_n) )
\end{align*}
then shows the claim.
\end{proof}
\begin{remark}
The conditions of Lemma~\ref{lem_qn} are satisfied for the example given in Section~\ref{sect_prelim_exp_obs} where $f$ is polynomial and the constraint operator is given by $\calD u = u|_{\Omega_0}$ for a subset $\Omega_0 \subseteq \Omega$ of positive measure.
%
\end{remark}
%
%
\subsection{Local error}\label{sect_lie_local}
In this subsection, we analyse the local error of Lie splitting, i.e., we compare the exact solution $u(t_{n+1})$ with $u_{n+1}$ given by one step of the Lie splitting with starting value $\tilde u_n := u(t_n)$.
The exact value is given by the variation-of-constants formula from Section~\ref{sect_prelim_solFormula}.
For this, we consider $f(u_n)$ as part of the right-hand side which leads to the implicit representation
\begin{align}
\label{eqn_exact_sln}
  u(t_{n+1})
  &= \calD^-\calG_n(\tau) + \ee^{\tau\calAo} \big( \tilde u_n - \calD^-\calG_n(0) \big)  \\
  &\qquad\qquad  + \int_0^\tau \ee^{(\tau-s)\calAo} \P \Big(\calF_n(s) + f(u(t_n+s)) + \calA\calD^-\calG_n(s) \Big) \ds. \notag
\end{align}
The solution of system \eqref{eqn_lie_A} in the time interval $[t_n, t_{n+1}]$ is given by
\[
  w_n(\tau)
  = w_n(0) + \int_0^\tau \Big( f(w_n(s)) - q_n(s) \Big) \ds
  = \tilde u_n + \tau \big( f(\tilde u_n) - q_n(0) \big) + O(\tau^2).
\]
Furthermore, the variation-of-constants formula yields for \eqref{eqn_lie_B} the representation
\[
  v_n(\tau)
  = \calD^-\calG_n(\tau) + \ee^{\tau\calAo} \P  w_n(\tau)
    + \int_0^\tau \ee^{(\tau-s)\calAo} \P \Big(\calF_n(s) + q_n(s) + \calA\calD^-\calG_n(s) \Big) \ds.
\]
Note that we have used here the formula from Remark~\ref{rem_inconsistent}, since we can not assume that the outcome of the nonlinear system \eqref{eqn_lie_A}, namely $w_n(\tau)$, is consistent with the constraint.
With $\tilde u_n - \calD^-\calG_n(0) = \P \tilde u_n$, the local error can be represented in the form
\begin{align*}
  \delta_{n+1}
  &= \cL_\tau \tilde u_n - u(t_{n+1}) \\
  &= \ee^{\tau\calAo} \P\big( w_n(\tau) -\tilde u_n \big)
       + \int_0^\tau \ee^{(\tau-s)\calAo} \P \Big( q_n(s) - f(u(t_n+s)) \Big) \ds + O(\tau^2) \\
  &= \tau \ee^{\tau\calAo} \P \big( f(\tilde u_n) - q_n(0) \big)
       + \int_0^\tau \ee^{(\tau-s)\calAo} \P \Big( q_n(s) - f(u(t_n+s)) \Big) \ds + O(\tau^2).
\end{align*}
With $\hat\psi(s) := q_n(s)-f(u(t_n+s))$ and $\psi(s) := \ee^{(\tau-s)\calAo} \P \hat \psi(s)$ the local error can be expressed in terms of a quadrature error in the following way
\begin{align*}
  \delta_{n+1}
  = - \tau \psi(0)
     + \int_0^\tau \psi(s) \ds + O(\tau^2).
\end{align*}
Note that Taylor expansion of $\psi$ yields the formula
\begin{align}
\label{eqn_lie_quadrature}
  \int_0^\tau \psi(s) \ds
  = \tau \psi(0) + \int_0^\tau \int_0^s \psi'(\xi) \dxi \ds.
\end{align}
In order to show the second-order property of the local error, it is sufficient to show the boundedness of $\psi'(\xi)$.
The derivative is given by
\[
  \psi'(\xi)
  = \ee^{(\tau-s)\calAo} \Big(-\calAo \P\hat\psi(s) + \P \hat\psi'(s) \Big).
\]
Note that $\hat\psi'(s)$ is bounded as well as $\calAo \P\hat\psi(s)$, since we have $\P\hat\psi(s) \in D(\calAo)$ by \eqref{eqn_lie_smoothness}.
In summary, we obtain $\delta_{n+1} = O(\tau^2)$ independent of the particular choice of the correction. 
%
%
\subsection{Reversed order}\label{sect_lie_rev}
Because of the projection of the initial data to the constraint manifold, we actually did not need the correction $q_n$ within the error analysis of Lie splitting.
However, if we reverse the order of the two subsystems, the situation changes slightly.
For this, assume that we first solve the linear PDAE
\begin{align}
\label{eqn_rev_A}
  \dot v_n - \calA v_n + \calD^- \lambda = \calF_n + q_n, \qquad
  \calD v_n = \calG_n, \qquad
  v_n(0) =  u_n
\end{align}
and then the nonlinear ODE
\begin{align}
\label{eqn_rev_B}
  \dot w_n = f(w_n) - q_n, \qquad
  w_n(0) = v_n(\tau).
\end{align}
For the analysis of the local error we take as initial value $\tilde u_n = u(t_n)$.
As before, we approximate the solution of \eqref{eqn_rev_B} by
\[
 w_n(\tau)
  = w_n(0) + \int_0^\tau \Big(  f(w_n(s)) - q_n(s) \Big)  \ds
  = v_n(\tau) + \tau \big[ f(v_n(\tau)) - q_n(0) \big] + O(\tau^2)
\]
and $v_n(\tau)$ by the variation-of-constants formula.
For the local error this results in
\begin{align*}
  \delta^\text{rev}_{n+1}
  &= \cL^\text{rev}_\tau \tilde u_n - u(t_{n+1}) \\
  &= v_n(\tau) + \tau \big[ f(v_n(\tau)) - q_n(0) \big] - u(t_{n+1}) + O(\tau^2) \\
  &= \tau \big[ f(v_n(\tau)) - q_n(0) \big]
    + \int_0^\tau \ee^{(\tau-s)\calAo} \P \Big( q_n(s) -  f(u(t_n+s)) \Big) \ds   +  O(\tau^2)   \\
  &= -\tau \hat\psi(0) + \tau \big[ f(v_n(\tau)) - f( \tilde u_n) \big]
    + \int_0^\tau \ee^{(\tau-s)\calAo} \P \hat\psi(s) \ds   +  O(\tau^2).
\end{align*}
Using
\[
  \tau \hat\psi(0)
  = \tau (I-\P) \hat\psi(0) + \underbrace{\tau (I-\ee^{\tau\calAo}) \P\hat\psi(0)}_{=O(\tau^2)} +\, \tau \ee^{\tau\calAo} \P\hat\psi(0),
\]
we split the local error into three parts.
First, we have the quadrature error as in the previous subsection which is of second order.
Second, we consider the projection error $\tau (I-\P) \hat\psi(0)$ for which we use the assumption on the correction $q_n$.
Applying the right-inverse $\calD^-$ to \eqref{eqn_lie_qn}, we obtain $(I-\P) \hat\psi(0) = O(\tau)$.
Finally, we need an estimate of the difference $v_n(\tau)-\tilde u_n$ for which we calculate
\[
  v_n(\tau)-\tilde u_n
  = (\ee^{\tau \calAo}-I) \big[ \tilde u_n - \calD^-\calG_n(0) \big] + O(\tau).
\]
%
For this, we have used that the integral term in the solution formula of $v_n(\tau)$ is of order $\tau$ as well as $\calG_n(\tau) - \calG_n(0) = O(\tau)$.
To show $v_n(\tau)-\tilde u_n = O(\tau)$, we need $\tilde u_n - \calD^-\calG_n(0) \in D(\calAo)$.
This is true, since $\tilde u_n$ is consistent such that $\tilde u_n - \calD^-\calG_n(0) \in X_0$ and $\calD^-\calG_n(0) \in D(\calA)$ by assumption (A5).
In summary, we obtain $\delta^\text{rev}_{n+1} = O(\tau^2)$ if assumption \eqref{eqn_lie_qn} on the correction is satisfied.
If this is not the case, then the calculation shows that
\[
  \delta^\text{rev}_{n+1}
  = \tau (I-\P) \hat\psi(0) + O(\tau^2).
\]
However, we will show that this still leads to a first-order method.
%
%
\subsection{Global error}\label{sect_lie_global}
In this final part on Lie splitting we prove that the method is first-order convergent no matter which order of the subsystems has been chosen.
We show this independently of the choice of the correction $q_n$, as long as \eqref{eqn_lie_smoothness} is satisfied.
Recall that we have shown in the previous two subsections
\[
  \delta_{n+1} =  O(\tau^2), \qquad\quad
  \delta^\text{rev}_{n+1} = \tau (I-\P) \hat\psi(0) + O(\tau^2).
\]
With these achievements on the local errors we prove the following convergence result.
\begin{theorem}[Convergence of Lie splitting]
Under the assumptions (A1)-(A6) given in Section~\ref{sect_prelim_ass} and a correction $q_n$ satisfying $q_n(s)-f(u(t_n+s)) \in D(\calA)$, Lie splitting is convergent of first order.
\end{theorem}
\begin{proof}
We consider Lie splitting in reversed order from Section~\ref{sect_lie_rev}, since the result is trivial if the local error is of second order.
Let $e_n:= u_n - u(t_n)$ denote the global error at time $t= t_n=n\tau$ where $u_n$ denotes the approximate solution $(\cL_\tau^\text{rev})^n \uo$.
Then, we obtain by the formulas for $v_n(\tau)$ and $w_n(\tau)$ 
the error recursion
\begin{align*}
  e_{n+1}
  = \cL^\text{rev}_\tau u_n - u(t_{n+1})
  &= \cL^\text{rev}_\tau u_n - \big( \cL^\text{rev}_\tau u(t_n) - \delta^\text{rev}_{n+1} \big) \\
  &= \ee^{\tau\calAo} \P e_n + \tau \big[ f(v_n(\tau)) - f(\tilde v_n(\tau)) \big] + \delta^\text{rev}_{n+1} + O(\tau^2),
\end{align*}
where $v_n$ and $\tilde v_n$ denote the solutions of \eqref{eqn_rev_A} with initial data $u_n$ and $\tilde u_n =u(t_n)$, respectively.
Moreover, the Lipschitz continuity of $f$ yields
\[
  \big\Vert f(v_n(\tau)) - f(\tilde v_n(\tau)) \big\Vert
  \le C\, \Vert v_n(\tau) - \tilde v_n(\tau) \Vert
  = C\,   \Vert \ee^{\tau\calAo} \P e_n \Vert.
\]
Note that the local errors $\delta^\text{rev}_k$ for $k\le n$ only occur with the projection $\P$ so that the first-order error terms vanish in this case.
The application of a discrete Gronwall argument then yields the claim.
\end{proof}
\begin{remark}
We emphasize once more that we do not require \eqref{eqn_lie_qn} in the convergence result.
However, numerical experiments show that a correction satisfying \eqref{eqn_lie_qn} may lead to better results,
especially for the reversed Lie splitting, see Section~\ref{sect_numerics}.
Furthermore, since we have not used the parabolic smoothing property \eqref{eqn_smoothing}, the convergence result is also true for strongly continuous semigroups.
\end{remark}

\section{Strang Splitting}\label{sect_strang}
This section is devoted to the convergence of Strang splitting for constrained diffusion-reaction systems.
For this, we still consider the assumptions from Section~\ref{sect_prelim_ass} which include that the operator $\calAo := \P \calA|_{X_0}$ generates an analytic semigroup.
As for the Lie splitting we introduce a correction term $q_n$.
However, in contrast to the previous section, we need $q_n$ as defined in \eqref{eqn_lie_qn} in order to guarantee the second-order convergence property.
%
%
\subsection{Modified Strang splitting}\label{sect_strang_alg}
We consider once more system \eqref{eqn_lie_gesamt} with initial condition $u(0)= \uo$.
To find an approximation of $u$ in $[t_n, t_{n+1}]$, we consider the following three subsystems.
First, we solve a linear PDAE on the interval $[0, \ttau]$, namely
\begin{align}
\label{eqn_strang_A}
  \dot v_n - \calA v_n + \calD^- \lambda = \calF_n + q_n, \qquad
  \calD v_n = \calG_n, \qquad
  v_n(0) = u_n.
\end{align}
As before, $q_n$ denotes the correction term of the current splitting step.
Second, we take $v_n(\ttau)$ as initial value for the nonlinear ODE including the reaction term on the time interval $[0, \tau]$, i.e.,
\begin{align}
\label{eqn_strang_B}
  \dot w_n = f(w_n) - q_n, \qquad
  w_n(0) = v_n(\ttau).
\end{align}
Finally, we consider again the linear PDAE but now in the time domain $[\ttau, \tau]$ and with an initial value coming from \eqref{eqn_strang_B},
\begin{align}
\label{eqn_strang_C}
  \dot v_n - \calA v_n + \calD^- \lambda = \calF_n + q_n, \qquad
  \calD v_n = \calG_n, \qquad
  v_n(\ttau) = w_n(\tau).
\end{align}
For the correction we assume, as in the previous section, \eqref{eqn_lie_smoothness} as well as \eqref{eqn_lie_qn}, cf.~the discussion on the definition of $q_n$ in Section~\ref{sect_lie_alg}.
This choice then ensures that Strang splitting is of second order.
For an estimate of the local error, we follow the procedure given in \cite{EinO15, EinO16ppt}.
%
%
\subsection{Local error}\label{sect_strang_local}
We analyse the error of one splitting step with (exact) initial value $v_n(0) = \tilde u_n := u(t_n)$.
Note that this implies that the value is consistent with the constraint.
The solution of system \eqref{eqn_strang_A} is given by the variation-of-constants formula of Section~\ref{sect_prelim_solFormula},
\[
  v_n(\ttau)
  = \calD^-\calG_n(\ttau) + \ee^{\frac \tau 2\calAo} \big( \tilde u_n - \calD^-\calG_n(0) \big)
     + \int_0^{\tau/2} \ee^{(\frac\tau 2-s)\calAo} \P\Big( \calF_n(s) + q_n(s) + \calA\calD^-\calG_n(s) \Big) \ds.
\]
The solution of the nonlinear system \eqref{eqn_strang_B} satisfies
\[
  w_n(\tau)
  = v_n(\ttau) + \int_0^\tau \dot w_n(s) \ds
  = v_n(\ttau) + \tau \big[ f(w_n(\ttau)) - q_n(\ttau) \big] + O(\tau^3).
\]
This is seen by applying the midpoint rule to the integral.
Before we consider the local error of Strang splitting, we give a preparatory lemma concerning the difference of $w_n(\ttau)$ and $u(t_n+\ttau)$.
\begin{lemma}[cf.~{\cite[Lem.~4.6]{EinO16ppt}}]
\label{lem_wn_un}
In addition to the assumptions (A1)-(A6) in Section~\ref{sect_prelim_ass} we assume that the correction $q_n$ satisfies  \eqref{eqn_lie_smoothness} and \eqref{eqn_lie_qn}.
Then, we get the estimate
\[
  \Vert w_n(\ttau) - u(t_n+ \ttau) \Vert \le C\, \tau^2.
\]
\end{lemma}
\begin{proof}
Since $w_n(\ttau)$ equals the outcome of the reversed Lie splitting with step size $\ttau$, we can use the previous proof.
Thus, assuming that \eqref{eqn_lie_qn} is satisfied, we obtain
\[
  w_n(\ttau) - u(t_n+ \ttau)
  = \cL^\text{rev}_{\tau/2} \tilde u_n - u(t_n+ \ttau)
  = O(\tau^2).
  \qedhere
\]
\end{proof}
For the local error of the splitting, we apply the variation-of-constants formula to system \eqref{eqn_strang_C}.
Since the initial value $w_n(\tau)$ will be inconsistent, in general, we use its projection on $X_0$, cf.~Remark~\ref{rem_inconsistent}.
We then obtain
\begin{align*}
  \calS_\tau \tilde u_n   
  &= v_n(\tau)\\
  &= \calD^-\calG_n(\tau) + \ee^{\frac \tau 2\calAo} \P w_n(\tau)
     + \int_{\tau/2}^\tau \ee^{(\tau - s)\calAo} \P \Big(\calF_n(s) + q_n(s) + \calA\calD^-\calG_n(s) \Big) \ds \\
  &= \calD^-\calG_n(\tau) + \ee^{\tau\calAo} \big( \tilde u_n - \calD^-\calG_n(0) \big)
     + \tau \ee^{\frac \tau 2\calAo} \P \big(  f(w_n(\tfrac \tau 2)) - q_n(\tfrac \tau 2) \big) \\
     &\qquad\qquad\qquad\quad + \int_0^\tau \ee^{(\tau -s)\calAo} \P \Big(\calF_n(s) + q_n(s) + \calA\calD^-\calG_n(s) \Big)\ds + O(\tau^3).
\end{align*}
Here, we have used the fact that $\P v_n(\ttau) = v_n(\ttau) - \calD^-\calG(\ttau)$.
Taking the difference with the exact solution $u(t_{n+1})$ which is given by \eqref{eqn_exact_sln} leads to the local error representation
\begin{align*}
  \delta_{n+1}
  &= \calS_\tau \tilde u_n - u(t_{n+1}) \\
  &= \tau \ee^{\frac \tau 2\calAo} \P \big(  f(w_n(\tfrac \tau 2)) - q_n(\tfrac \tau 2) \big)
     + \int_0^\tau \ee^{(\tau - s)\calAo} \P \Big(q_n(s) - f(u(t_n+s)) \Big) \ds + O(\tau^3).
\end{align*}
Calling again $\hat\psi(s)= q_n(s)-f(u(t_n+s))$, we can rewrite this as
\begin{align*}
  \delta_{n+1}
  &= \underbrace{\tau \ee^{\frac \tau 2\calAo}\P \big( f(w_n(\ttau)) - f(u(t_n+\ttau)) \big)}_{\textcircled{\scriptsize \texttt{1}}}
      \\
     &\qquad\qquad\qquad \underbrace{-\, \tau \ee^{\frac \tau 2\calAo} \P \hat \psi(\ttau)
                   	  	 + \int_0^\tau \ee^{(\tau - s)\calAo} \P \hat \psi(s) \ds }_{\textcircled{\scriptsize \texttt{2}}}  +\, O(\tau^3).
\end{align*}
%
To estimate {\textcircled{\scriptsize \texttt{1}}} we use the Lipschitz continuity of $f$, the continuity of the projection $\P$, and Lemma~\ref{lem_wn_un} which includes condition \eqref{eqn_lie_qn},
\[
  \big\Vert \tau \ee^{\frac \tau 2\calAo} \P \big(  f(w_n(\tfrac \tau 2)) - f(u(t_n+\ttau)) \big) \big\Vert
  \le \tau\, C\, \Vert w_n(\tfrac \tau 2) - u(t_n+ \tfrac \tau 2) \Vert
  \le C\, \tau^3. 
\]
%
For an estimate of {\textcircled{\scriptsize \texttt{2}}} we proceed as in \cite{EinO16ppt} and use the Peano kernel representation of the error of the midpoint rule.
However, because of the included projection, we do not need the assumption on $q_n$ at this point.
With $\psi(s) = \ee^{(\tau-s)\calAo} \P \hat\psi(s)$ we have
\begin{align*}
 \int_0^\tau \psi(s)\ds - \tau \psi(\tfrac \tau 2)
 = \int_0^{\tau/2} \frac{s^2}{2} \psi''(s)\ds + \int_{\tau/2}^\tau \frac{(\tau-s)^2}{2} \psi''(s)\ds.
\end{align*}
For the claimed $O(\tau^3)$ property it is sufficient to show that $\psi''(s)$ is bounded,
\begin{align*}
  \psi''(s)
  &= \calAo \ee^{(\tau-s)\calAo} \calAo \P\hat\psi(s) - 2 \calAo \ee^{(\tau-s)\calAo} \P\hat\psi'(s) + \ee^{(\tau-s)\calAo} \P\hat\psi''(s) \\
  &= \calAo \Big[ \ee^{(\tau-s)\calAo} \calAo \P\hat\psi(s) - 2 \ee^{(\tau-s)\calAo} \P\hat\psi'(s) \Big] + \ee^{(\tau-s)\calAo} \P\hat\psi''(s).
\end{align*}
Clearly, the term including $\P\hat\psi''(s)$ is bounded, since $\P$ is a continuous projection.
Also $\ee^{(\tau-s)\calAo} \P\hat\psi'(s)$ is bounded and the preceding operator $\calAo$ can be compensated by parabolic smoothing in the error recursion (except for the local error of the final step).
Finally, we make use of $\P\hat\psi(s) \in D(\calAo)$ 
which implies that $\ee^{(\tau-s)\calAo} \calAo \P\hat\psi(s)$ is bounded as well.

Together with the parabolic smoothing property \eqref{eqn_smoothing} this shows
 \begin{align}
\label{eqn_local_error}
  \Vert \ee^{t\calAo} \delta_{n+1} \Vert
  \le C \frac{\tau^3}{t},\quad t>0.
\end{align}
Moreover, proceeding as in the proof of Lie splitting, \eqref{eqn_lie_quadrature} and $\psi(0) = \psi(\ttau) + O(\tau)$ show $\delta_{n+1} = O(\tau^2)$.
This bound is needed for the final step of the error recursion.
%
%
\subsection{Global error}\label{sect_strang_global}
With the representation of $\calS_\tau$ we obtain the error recursion
\begin{align*}
  e_{n+1}
  &= u_{n+1} - u(t_{n+1}) \\
  &= \calS_\tau u_n - \big(  \calS_\tau u(t_n)  -\delta_{n+1} \big) \\
  &= \ee^{\tau\calAo} e_n
     + \tau \ee^{\frac \tau 2\calAo} \P \big(  f(w_n(\tfrac \tau 2)) - f(\tilde w_n(\tfrac \tau 2)) \big)
     + \delta_{n+1} + O(\tau^3).
\end{align*}
Here, $w_n$ and $\tilde w_n$ denote the solutions of \eqref{eqn_strang_B} based on the initial data $u_n$ and $\tilde u_n =u(t_n)$, respectively.
Because of the Lipschitz continuity of $f$ and
\[
  w_n(\tfrac \tau 2) - \tilde w_n(\tfrac \tau 2)
  = \ee^{\frac \tau 2\calAo}  e_n              
    + \frac \tau 2 \big[ f(u_n) - f( u(t_n)) \big] + O(\tau^2)
\]
the global error can be written as
\begin{align*}
  e_{n+1}
  = \ee^{\tau\calAo} e_n + \tau E_n + \delta_{n+1} + O(\tau^3)
\end{align*}
with $E_n$ satisfying the uniform bound $\Vert E_n\Vert \le C (1+\tau) \Vert e_n\Vert$.
The results on the local error and in particular \eqref{eqn_local_error} imply that
\[
  \Vert \delta_{n+1} \Vert \le C\, \tau^2, \qquad
  \Vert \ee^{k\tau\calAo} \delta_{n+1-k}\Vert \le C \frac{\tau^2}{k}.
\]
Thus, solving the error recursion, we obtain
\begin{align}
\label{eqn_strang_globalErr}
  \Vert e_n \Vert
  \le  C\,  \Vert e_0 \Vert                               
     + C\, \tau (1+\tau) \sum_{k=0}^{n-1} \Vert e_k\Vert  
     + C\, \tau^3 \sum_{k=1}^{n-1} \frac{1}{k\tau}        
     + C\, \tau^2.                                        
\end{align}
Since we assume $e_0=0$, a Gronwall argument and the fact that the third term on the right-hand side of \eqref{eqn_strang_globalErr} is bounded by $C \tau^2 (1+ \left|\log \tau\right|)$ yields the following result for Strang splitting.
\begin{theorem}[Convergence of Strang splitting]
Given assumptions (A1)-(A6) of Section~\ref{sect_prelim_ass} and a correction $q_n$ which satisfies \eqref{eqn_lie_smoothness} as well as \eqref{eqn_lie_qn}, Strang splitting is convergent of second order.
More precisely, we have the uniform estimate
\[
  \Vert u(t_n) - u_n \Vert \le C\tau^2 (1+\left|\log \tau\right|)
\]
on bounded time intervals $0 \le t_n = n \tau \le T$.
\end{theorem}
With this result we close this section and advance to the numerical experiments.

\section{Numerical Examples}\label{sect_numerics}
This section is devoted to the numerical verification of the analytical results obtained in the previous two sections.
We consider three of the examples introduced in Section~\ref{sect_prelim_exp}.
%
%
\subsection{Weighted integral mean}\label{sect_numerics_int}
We consider the semilinear heat equation in $\Omega = (0,1)$ and time interval $[0,0.1]$ with a polynomial nonlinearity as in Section~\ref{sect_prelim_exp_int}.
The resulting equations are given by
\[
  \dot u - \tfrac{1}{10}\Delta u - u^2 + \calD^- \lambda = 0, \qquad
  \calD u(t) := \int_0^1 u(t,x) \sin(\pi x) \dx = \calG(t),
\]
where the right-hand side is given by $\calG(t) = t$.
In addition, we prescribe homogeneous Dirichlet boundary conditions, i.e., $u(t,0) = u(t,1) = 0$, and the following initial condition
\[
  \uo(x) = \sin(2 \pi x)^3, \qquad
  \calD \uo = 0 = \calG(0).
\]
We compare Lie and Strang splitting with and without the correction term which is given by $q_n = f(u_n)$ as proposed in \eqref{eqn_qn_choiceA}.
For the numerical solution we have used a second-order finite difference approximation of the Laplacian with $500$ grid points.
The results for Lie splitting are given in Table~\ref{tab_Integral_Lie} and show the predicted first-order convergence with and without the correction term.
Here, the error without correction is even smaller than with $q_n=f(u_n)$.
In contrast, the reversed Lie splitting from Section~\ref{sect_lie_rev} is more accurate by a factor of $2$ with the correction.
%
Table~\ref{tab_Integral_Strang} shows the expected convergence rates for Strang splitting, namely first-order convergence without the correction and second-order convergence with the correction.
\begin{table}[]
\centering
\caption{Convergence history of Lie splitting with the constraint on the weighted integral mean (Section~\ref{sect_numerics_int}).}
\label{tab_Integral_Lie}
\begin{tabular}{l | cc | cc}
			&  \multicolumn{2}{c|}{Lie}  & \multicolumn{2}{c}{Lie with correction} 	\\ \hline
step size 	& $\ell^\infty$ error & order & $\ell^\infty$ error & order 	\\ \hline			
2.000e-02   &  8.895e-04  & --       &  2.028e-03  & --          			\\
1.000e-02   &  4.362e-04  & 1.03     &  9.623e-04  & 1.08      				\\
5.000e-03   &  2.161e-04  & 1.01     &  4.687e-04  & 1.04      				\\
2.500e-03   &  1.076e-04  & 1.01     &  2.313e-04  & 1.02      				\\
1.250e-03   &  5.368e-05  & 1.00     &  1.149e-04  & 1.01      				\\
6.250e-04   &  2.681e-05  & 1.00     &  5.727e-05  & 1.00      				\\
\end{tabular}
\end{table}
\begin{table}[]
\centering
\caption{Convergence history of Strang splitting in various norms with constraint on the weighted integral mean (Section~\ref{sect_numerics_int}).}
\label{tab_Integral_Strang}
\begin{tabular}{l | cc | cc | cc}
			&  \multicolumn{2}{c|}{Strang}  &  \multicolumn{2}{c|}{Strang}  & \multicolumn{2}{c}{Strang with correction} 	\\ \hline
step size 	& $\ell^\infty$ error & order   & $\ell^2$ error & order & $\ell^\infty$ error & order 	\\ \hline			
2.000e-02   &  2.212e-04  	& --   			& 1.386e-04		& --   			&  5.657e-05  & --         		\\
1.000e-02   &  1.035e-04  	& 1.10     		& 6.824e-05		& 1.02			&  1.277e-05  & 2.15      		\\
5.000e-03   &  5.006e-05  	& 1.05     		& 3.390e-05		& 1.01			&  3.061e-06  & 2.06      		\\
2.500e-03   &  2.534e-05  	& 0.98     		& 1.690e-05 	& 1.00			&  7.496e-07  & 2.03      		\\
1.250e-03   &  1.275e-05  	& 0.99     		& 8.439e-06		& 1.00			&  1.855e-07  & 2.02      		\\
6.250e-04   &  6.396e-06  	& 1.00     		& 4.217e-06		& 1.00			&  4.613e-08  & 2.01      		\\
\end{tabular} 
\end{table}
%

Finally, we take a look at the convergence orders of the local errors.
It is clear from the analysis of Section~\ref{sect_lie} that Lie splitting converges of order one but the local errors may be of first or second order depending on the chosen sequence of the subsystems.
This can also be observed for the present example, cf.~Table~\ref{tab_Integral_Local}.
%
There, we also consider the order of the reversed Strang splitting, meaning that we first solve the nonlinear ODE, then the PDAE, and finally the ODE again.
This does not effect the global error, but the local error loses one order.
\begin{table}[]
\makegapedcells
\centering
\caption{Comparison of the order of the local $\ell^\infty$ errors of Lie and Strang splitting depending on the chosen sequence of subsystems and the choice of $q_n$.}
\label{tab_Integral_Local}
\begin{tabular}{l || cc | cc || cc | cc}
			&  \multicolumn{2}{c|}{Lie} & \multicolumn{2}{c||}{reversed Lie}	&  \multicolumn{2}{c|}{Strang} & \multicolumn{2}{c}{reversed Strang} \\ \hline
\diaghead{\theadfont laengevonCell} {order} {$q_n$}
			&\ \ $0$	& $f(u_n)$	&\ \ $0$	& $f(u_n)$	&\ \ $0$	& $f(u_n)$	&\ \ $0$	& $f(u_n)$ 		\\ \hline			
local error 		&\ \ 2  	& 2      	&\ \ 1  	& 2 		&\ \ 2		& 3			&\ \ 1		& 2    			\\
global error 		&\ \ 1  	& 1      	&\ \ 1  	& 1 		&\ \ 1		& 2			&\ \ 1		& 2      			
\end{tabular}
\end{table}
%
%
\subsection{Specification on a subset}\label{sect_numerics_obs}
As second example we consider the same PDE with the constraint that $u$ is prescribed on the subinterval $\Omega_0 = [0.5,\ 0.7]$, cf.~Section~\ref{sect_prelim_exp_obs}.
As initial condition we take
\[
  \uo(x) = \sin(\pi x)\cdot \big(1+ \cos(7\pi x)\big),
\]
whereas the constraint is given by
\[
  \calD u(t) = u(t)|_{\Omega_0} = \calG(t) := (1+2t)\, \calG(0),\qquad
  \calG(0) := \uo|_{\Omega_0}.
\]
For this example, the use of a correction satisfying \eqref{eqn_lie_qn} improves the results of Lie splitting by a factor of $10$.
In the following, we compare different kinds of corrections $q_n$ and the influence on the convergence of Strang splitting, see Figure~\ref{fig_exp2_strang}.
\begin{figure}[thbp]
	\begin{center}
    \includegraphics{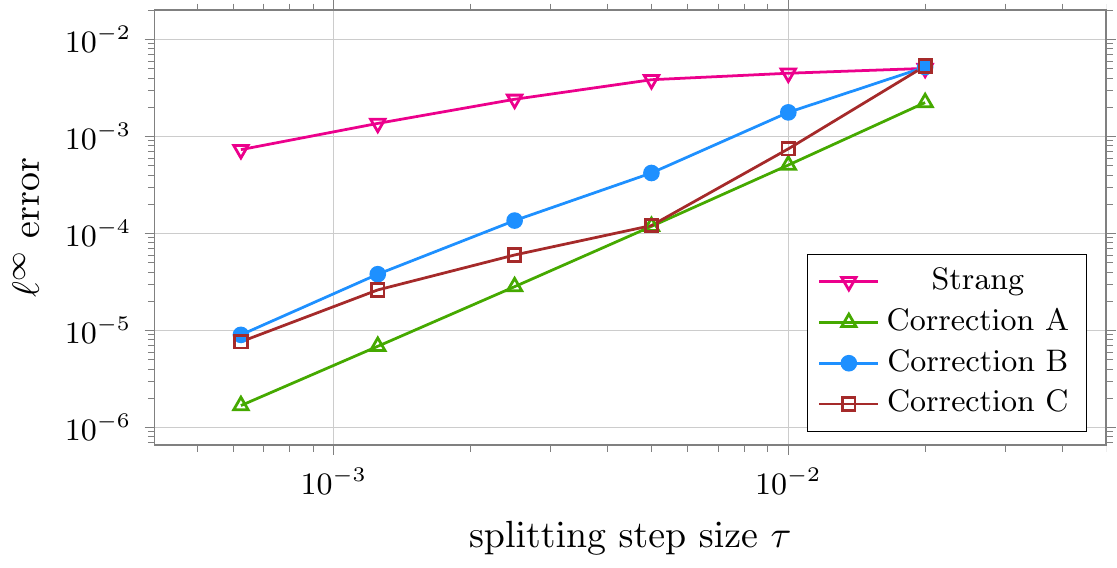}
	\end{center}
\caption{Comparison of different corrections $q_n$ for Strang splitting applied to the problem of Section~\ref{sect_numerics_obs}.}
\label{fig_exp2_strang}
\end{figure}
%
The first correction, called Correction A in the figure, is given by $q_n = f(u_n)$ as proposed in \eqref{eqn_qn_choiceA}.
This means that we simply apply the nonlinearity to the current approximation.
Second, we take the same correction with an additional perturbation outside of $\Omega_0$ (Correction B).
More precisely, we consider
\[
  q_n = f(u_n)+ p, \quad
  p(x) = \begin{cases} \phantom{-} \sin(2 \pi x) \cos(42 \pi x), & x < 0.5  \\
                                 \phantom{-} 0, & 0.5 \le x \le 0.7 \\
                                 -\sin\big(\frac{10}{3} \pi (x-0.7) \big) \cos\big(70 \pi (x-0.7)\big), & x > 0.7 \end{cases}.
\]
Note that this has no influence on condition \eqref{eqn_lie_qn} and thus does not change the order of convergence.
Nevertheless, the errors are larger by a factor of about $5$.
Finally, we take as third approach $q_n= f(\calD^-\calG(t_n))$ as mentioned in \eqref{eqn_qn_choiceB}.
Also here we observe second-order convergence (Correction C).
%
%
\subsection{Coupled mechanical system}\label{sect_numerics_spring}
We consider the example from Section~\ref{sect_prelim_exp_spring}, i.e., a string (fixed at its endpoints) which is coupled with a nonlinear spring-damper system, see Figure~\ref{fig_spring} for an illustration.
The computations use a spatial grid with $250$ nodes. The parameters are
\[
  c = 0.5, \qquad    
  d_1 = 10, \qquad   
  d_2 = 3, \qquad    
  k_0 = 100, \qquad
  a = 10, \qquad
  \Omega_0 = [0.65,\ 0.7].
\]
Recall that we consider a softening spring with nonlinearity $k(q)= k_0(1 - a^2q^2)\, q$.  
As described in Section~\ref{sect_prelim_exp_spring}, the constraints are given by
\[
  \calB u(t) := u(t)|_{\Omega_0} = q(t) + 0.05, \qquad
  \calB v(t) = p(t).
\]
As initial conditions we take
\[
  \uo(x) = \frac 15\, \P^u \sin(2\pi x)^3,\quad v_0(x) = \calB^-(0.5),  \qquad
  q_0 = -0.05,\quad p_0 = 0.5.
\]
Here, $\calB^-$ denotes the right-inverse of $\calB$ and $\P^u$ is the projection onto the kernel of $\calB$.
Thus, the given initial data is consistent with the constraint.

As expected, Lie splitting converges with order one whereas Strang splitting converges with order one or two depending on the implemented correction term.
The errors for Strang splitting are given in Table~\ref{tab_Mechanical_Strang}.
\begin{table}[b]
\centering
\caption{Convergence history of Strang splitting in different norms for the coupled mechanical system (Section~\ref{sect_numerics_spring}).}
\label{tab_Mechanical_Strang}
\begin{tabular}{l | cc | cc | cc}
			&  \multicolumn{2}{c|}{Strang}  &  \multicolumn{2}{c|}{Strang}  & \multicolumn{2}{c}{Strang with correction} 	\\ \hline
step size 	& $\ell^\infty$ error & order   & $\ell^2$ error & order & $\ell^\infty$ error & order 	\\ \hline			
4.000e-02   &  1.820e-02  	& --   			& 	4.680e-03	& --   			&  1.339e-02   & --         		\\
2.000e-02   &  8.281e-03  	& 1.14   		& 	1.757e-03	& 1.41   			&  6.101e-03   & 1.13         		\\
1.000e-02   &  5.616e-03  	& 0.56     		& 	1.293e-03	& 0.44			&  1.641e-03   & 1.89      		\\
5.000e-03   &  2.786e-03  	& 1.01     		& 	6.566e-04	& 0.98			&  2.851e-04   & 2.53      		\\
2.500e-03   &  1.414e-03  	& 0.98     		&  	3.338e-04	& 0.98			&  7.697e-05   & 1.89      		\\
1.250e-03   &  7.123e-04  	& 0.99     		& 	1.684e-04	& 0.99			&  1.946e-05   & 1.98      		\\
\end{tabular} 
\end{table}
%

%
\section{Conclusion}\label{sect_conclusion}
In this paper, we analysed splitting methods for constrained diffusion-reaction systems.
Splitting methods proved to be advantageous as they allowed us to solve a {\em linear} constrained system and a nonlinear ODE separately.
With the help of a correction term, which we added in each step, we were able to overcome the order reduction of Strang splitting.
Although the numerical experiments show the benefits of an appropriate correction $q_n$, the analysis has shown that Lie splitting is convergent of first order also without any correction.
For Strang splitting, however, the correction is necessary to guarantee the second-order convergence.
%

\begin{thebibliography}{HLW10}

\bibitem[Alt15]{Alt15}
R.~Altmann.
\newblock {\em {R}egularization and {S}imulation of {C}onstrained {P}artial
  {D}ifferential {E}quations}.
\newblock PhD thesis, Technische Universit{\"a}t Berlin, 2015.

\bibitem[Des01]{Des01}
S.~Descombes.
\newblock Convergence of a splitting method of high order for
  reaction-diffusion systems.
\newblock {\em Math. Comp.}, 70(236):1481--1501, 2001.

\bibitem[EM13]{EmmM13}
E.~Emmrich and V.~Mehrmann.
\newblock Operator differential-algebraic equations arising in fluid dynamics.
\newblock {\em Comput. Methods Appl. Math.}, 13(4):443--470, 2013.

\bibitem[EN00]{EngN00}
K.-J. Engel and R.~Nagel.
\newblock {\em One-parameter semigroups for linear evolution equations}.
\newblock Springer, New York, 2000.

\bibitem[EO15]{EinO15}
L.~Einkemmer and A.~Ostermann.
\newblock Overcoming order reduction in diffusion-reaction splitting. {P}art 1:
  {D}irichlet boundary conditions.
\newblock {\em SIAM J. Sci. Comput.}, 37(3):A1577--A1592, 2015.

\bibitem[EO16]{EinO16ppt}
L.~Einkemmer and A.~Ostermann.
\newblock Overcoming order reduction in diffusion-reaction splitting. {P}art 2:
  oblique boundary conditions.
\newblock Technical report, 2016.
\newblock arXiv:1601.02288.

\bibitem[FK64]{FujK64}
H.~Fujita and T.~Kato.
\newblock On the {N}avier-{S}tokes initial value problem. {I}.
\newblock {\em Arch. Rational Mech. Anal.}, 16:269--315, 1964.

\bibitem[Hen81]{Hen81}
D.~Henry.
\newblock {\em Geometric theory of semilinear parabolic equations}.
\newblock Springer, Berlin, 1981.

\bibitem[HKO12]{HanKO12}
E.~Hansen, F.~Kramer, and A.~Ostermann.
\newblock A second-order positivity preserving scheme for semilinear parabolic
  problems.
\newblock {\em Appl. Numer. Math.}, 62(10):1428--1435, 2012.

\bibitem[HLW10]{HaiLW10}
E.~Hairer, C.~Lubich, and G.~Wanner.
\newblock {\em Geometric numerical integration}.
\newblock Springer, Heidelberg, second edition, 2010.

\bibitem[HO09]{HanO09}
E.~Hansen and A.~Ostermann.
\newblock High order splitting methods for analytic semigroups exist.
\newblock {\em BIT}, 49(3):527--542, 2009.

\bibitem[HV95]{HunV95}
W.~Hundsdorfer and J.~G. Verwer.
\newblock A note on splitting errors for advection-reaction equations.
\newblock {\em Appl. Numer. Math.}, 18(1-3):191--199, 1995.

\bibitem[HV03]{HunV03}
W.~Hundsdorfer and J.~Verwer.
\newblock {\em Numerical solution of time-dependent
  advection-diffusion-reaction equations}.
\newblock Springer, Berlin, 2003.

\bibitem[KM06]{KunM06}
P.~Kunkel and V.~Mehrmann.
\newblock {\em Differential-algebraic equations: analysis and numerical
  solution}.
\newblock European Mathematical Society (EMS), Z\"urich, 2006.

\bibitem[LMT13]{LamMT13}
R.~Lamour, R.~M{\"a}rz, and C.~Tischendorf.
\newblock {\em Differential-algebraic equations: a projector based analysis}.
\newblock Springer, Heidelberg, 2013.

\bibitem[Lun95]{Lun95}
A.~Lunardi.
\newblock {\em Analytic semigroups and optimal regularity in parabolic
  problems}.
\newblock Birkh\"auser/Springer, Basel, 1995.

\bibitem[Meh15]{Meh15}
V.~Mehrmann.
\newblock Index concepts for differential-algebraic equations.
\newblock In B.~Engquist, editor, {\em Encyclopedia of Applied and
  Computational Mathematics}, pages 676--681. Springer, Berlin, 2015.

\bibitem[MQ02]{McLQ02}
R.~I. McLachlan and G.~R.~W. Quispel.
\newblock Splitting methods.
\newblock {\em Acta Numerica}, 11:341--434, 2002.

\bibitem[Paz83]{Paz83}
A.~Pazy.
\newblock {\em Semigroups of linear operators and applications to partial
  differential equations}.
\newblock Springer, New York, 1983.

\bibitem[Sim98]{Sim98}
B.~Simeon.
\newblock {DAE}s and {PDE}s in elastic multibody systems.
\newblock {\em Numer. Algorithms}, 19:235--246, 1998.

\bibitem[Tay00]{Tay00}
M.E. Taylor.
\newblock Incompressible fluid flows on rough domains.
\newblock In A.~V. Balakrishnan, editor, {\em Semigroups of operators: theory
  and applications}, pages 320--334. Birkh\"auser, Basel, 2000.

\bibitem[You13]{You13}
I.~Yousept.
\newblock Optimal control of quasilinear {$H(\text{curl})$}-elliptic partial
  differential equations in magnetostatic field problems.
\newblock {\em SIAM J. Control Optim.}, 51(5):3624--3651, 2013.

\end{thebibliography}

\end{document}